\documentclass[11pt]{article}
\usepackage{amsmath}
\usepackage{amssymb}
\usepackage{amsthm}
\usepackage{cite}
\usepackage{mathrsfs,geometry,xcolor}
\usepackage{kbordermatrix}
\usepackage{enumerate}
\usepackage{bm}
\usepackage{blkarray}
\usepackage{tikz}

\numberwithin{figure}{section}
\numberwithin{table}{section}
\numberwithin{equation}{section}

\makeatletter
\newcommand\xleftrightarrow[2][]{\ext@arrow 0099{\longleftrightarrowfill@}{#1}{#2}}
\def\longleftrightarrowfill@{\arrowfill@\leftarrow\relbar\rightarrow}
\makeatother

\usepackage{tikz}
\usetikzlibrary{arrows,shapes,quotes}
\usetikzlibrary{patterns,decorations.pathreplacing}

\usepackage{authblk}

\numberwithin{table}{section}
\numberwithin{equation}{section}

\theoremstyle{plain}
\newtheorem{theorem}{Theorem}
\newtheorem{proposition}{Proposition}
\newtheorem{corollary}{Corollary}

\theoremstyle{definition}
\newtheorem{definition}{Definition}
\newtheorem{example}{Example}
\newtheorem{remark}{Remark}

\usepackage{authblk}
 \author[1,*]{ \textbf{Noel T. Fortun}}
 \author[1,2,3,4]{\textbf{Eduardo R. Mendoza}}

\affil[1]{\small \textit{Mathematics and Statistics Department, De La Salle University, Manila  0922, Philippines}}
\affil[2]{\textit{Center for Natural Sciences and Environmental Research, De La Salle University, Manila  0922, Philippines}}
\affil[3]{\textit{Max Planck Institute of Biochemistry, Martinsried near Munich, Germany}}
\affil[4]{\textit{Faculty of Physics, Ludwig Maximilian University, Munich 80539, Germany}}
\affil[*]{Corresponding author: \texttt{noel.fortun@dlsu.edu.ph}}

\title{\vspace{3.5cm}\textbf{Absolute concentration robustness in power law kinetic systems}}

\date{}

\begin{document}
\maketitle
\thispagestyle{empty}
\begin{abstract}
Absolute concentration robustness (ACR) is a condition wherein a species in a chemical kinetic system possesses the same value for any positive steady state the network may admit regardless of initial conditions. Thus far, results on ACR center on chemical kinetic systems with deficiency one.  
In this contribution, we use the idea of dynamic equivalence of chemical reaction networks to derive novel results that guarantee ACR for some classes of power law kinetic systems with deficiency zero. Furthermore, using network decomposition, we identify ACR in higher deficiency networks (i.e. deficiency $\geq 2$) by considering the presence of a low deficiency subnetwork with ACR.  Network decomposition also enabled us to recognize and define a weaker form of concentration robustness than ACR, which we named as `balanced concentration robustness'. Finally, we also discuss and emphasize our view of ACR as a primarily kinetic character rather than a condition that arises from structural sources. \\
\linebreak
\end{abstract}
\baselineskip=0.30in

\newpage
\section{Introduction}

A network is said to exhibit robustness if it maintains its function despite changes in environmental or structural conditions \cite{KITANO2004,KITANO2007}. As it is required for homeostasis and adaptive responses to environmental disruptions, robustness becomes fundamental and ubiquitous in many biological processes \cite{SF2010,KITANO2004,BLANCHINI2011,
BARKAI1997}. A class of robust behavior known as ``concentration robustness'' concerns the invariance of some quantity involving the concentrations of the different species in a network for any steady state \cite{DEXTER2015}. 

Of particular interest is the concentration robustness property called \textit{absolute concentration robustness} (ACR), which was first introduced by Shinar and Feinberg in their influential paper published in \textit{Science} \cite{SF2010}. A system possesses this feature if it admits at least one positive steady state and the concentration of a particular species in the system has the same value in every positive steady state set by parameters. The work of Shinar and Feinberg centered on a mathematical theorem that specifies a  large class of mass action systems that are absolute concentration robust. Interestingly, this theorem provides sufficient  conditions that are apparently structural in nature. 

Specifically, they stated their result around a structural index called the \textit{deficiency} (denoted by $\delta$), which measures the amount of  `linear independence' among the reactions of the network \cite{SF2011}. The theorem is stated as follows: \textit{Consider a mass action system that admits a positive steady state. Suppose that (i) the deficiency of the network is one, and (ii) 
there are nonterminal complexes which differ only in the species $X$. Then the system has ACR in species $X$}.  

In our previous work \cite{FLRM2020}, we showed that this result easily extends to kinetic systems more general than mass action systems 
namely, \textit{power law kinetic systems with reactant-determined interactions} (denoted by ``PL-RDK''). For PL-RDK systems, the kinetic order vectors of reactions with the same reactant complexes are identical. The \textit{Shinar-Feinberg Theorem on ACR for PL-RDK systems}  retains the deficiency one condition  but replaces the last criterion by considering the kinetic order differences of the species. Our result specifies that under the same deficiency one assumption, and the criterion that there are nonterminal complexes whose kinetic order of its species differ only in $X$, the PL-RDK system that admits a positive equilibrium exhibits ACR in $X$. 

In this contribution, we explore ACR as a dynamical property that is conserved under \textit{dynamic equivalence}. Two different chemical reaction networks with the same set of kinetics are dynamically equivalent if they generate the same set of ordinary differential equations. Significantly, this approach has led us to derive novel results on ACR  for deficiency zero PL-RDK systems and for a class of power law kinetic systems that are non-PL-RDK (denoted as ``PL-NDK").

In addition to dynamic equivalence, this contribution applies useful techniques in network decomposition to establish ACR. This is particularly relevant in detecting ACR in systems where the underlying chemical reaction networks have higher deficiency (i.e. $\delta \geq 2$). The concept of independent decompositions \cite{FEIN1987} has enabled us to identify ACR in larger networks through the presence of a low deficiency ($\delta \leq 1$) subnetwork with ACR as a ``building block.'' The key result used is a known theorem \cite{FEIN1987} that relates independent decomposition with the set of positive equilibria of a system. 

In an analogous approach, incidence independent decompositions \cite{FML2020} of larger networks with low deficiency subnetwork exhibiting ACR are also investigated. This effort has led us to identify another type of concentration robustness that is weaker than ACR. We call this property as \textit{balanced concentration robustness} (BCR). A system displays BCR in a species $X$ if it has complex balanced steady states and the value of $X$ is the same for any set of complex balanced steady states the system may admit. Using a theorem that relates incidence independent decompositions with the set of complex balanced equilibria of a system \cite{FML2020}, this work generates a new result that guarantees the presence of BCR for larger networks.  

Finally, this work provides a discussion that emphasizes the primarily kinetic property of ACR. This perspective is a shift from our usual view that ACR, as a system property, is induced by ``structural sources."

This paper is outlined as follows. Section 2 reviews and assembles fundamental ideas and results in chemical reaction network theory that are relevant for later sections. Section 3 presents the ACR theorem for deficiency zero PL-RDK systems and for a class of PL-NDK systems. In Section 4, we employ decomposition theory to identify large classes of PLK systems,  including such with higher deficiency, that possess ACR or BCR.  Section 5 discusses our view that ACR is a primarily kinetic property of a chemical kinetic system. Section 6 summarizes our results and outlines perspectives for future work. Lastly, the discussion in Appendix provides the adaptation of the proof presented in \cite{FLRM2020} for  deficiency zero PL-RDK networks.

\section{Fundamentals of chemical reaction networks theory}\label{sec:prelims}
We review notions and results (taken from \cite{AJMSM2015,TAM2018}) that are pertinent in understanding the results in this work. Some fundamental concepts introduced by Feinberg \cite{FEIN1979,FEIN1995} are also reviewed. 

\textbf{Notation:}  We denote the real numbers by $\mathbb{R}$, the non-negative real numbers by $\mathbb{R}_{\geq0}$, and the positive real numbers by $\mathbb{R}_{>0}$.  Objects in the reaction systems are viewed as members of vector spaces. Suppose $\mathscr{I}$ is a finite index set. By $\mathbb{R}^\mathscr{I}$, we mean the usual vector space of real-valued functions with domain $\mathscr{I}$.  For $x \in \mathbb{R}^\mathscr{I}$, the $i^\text{th}$ coordinate of $x$ is denoted by $x_i$, where $i \in \mathscr{I}$. The sets $\mathbb{R}_{\geq 0}^\mathscr{I}$ and $\mathbb{R}_{>0}^\mathscr{I}$ are called the \textit{non-negative} and \textit{positive orthants} of $\mathbb{R}^\mathscr{I}$, respectively. Addition, subtraction, and scalar multiplication in $\mathbb{R}^\mathscr{I}$are defined in the usual way. If $x \in \mathbb{R}_{>0}^\mathscr{I}$ and $y \in \mathbb{R}^\mathscr{I}$, we define $x^y \in \mathbb{R}_{>0}$ by
$
x^y= \prod_{i \in \mathscr{I}} x_i^{y_i} .
$
By the \textit{support} of $x \in \mathbb{R}^\mathscr{I}$, denoted by $\text{supp } x$, we mean the subset of $\mathscr{I}$ assigned with non-zero values by $x$. That is,
$
\text{supp } x := \{ i \in \mathscr{I} | x_i \neq 0 \}.
$ 
Finally, for integers $a$ and $b$, let $\overline{a,b}= \{ j \in \mathbb{Z} | a \leq j \leq b \}$. 

\subsection{Structure of chemical reaction networks}
A chemical reaction network (CRN) is a system of interdependent chemical reactions. Each reaction is represented as an ordered pair of vectors, called complexes, of chemical species. The interdependence of the reactions results in the description of the network as a directed graph (or digraph). 

\begin{definition}
A \textbf{chemical reaction network} (CRN) $\mathscr{N}$ is a triple $(\mathscr{S},\mathscr{C},\mathscr{R})$ of three finite sets:
\begin{enumerate}
\item a set $\mathscr{S}= \{X_1, X_2, \dots, X_m \}$ of \textbf{species};
\item a set $\mathscr{C} \subset \mathbb{R}^\mathscr{S}_{\geq 0}$ of \textbf{complexes};
\item a set $\mathscr{R} = \{R_1, R_2, \dots, R_r \}\subset \mathscr{C} \times \mathscr{C}$ of \textbf{reactions} such that $(y,y) \notin \mathscr{R}$ for any $y \in \mathscr{C}$, and  for each $y \in \mathscr{C}$, there exists $y' \in \mathscr{C}$ such that either $(y,y') \in \mathscr{R}$ or  $(y',y) \in \mathscr{R}$.
\end{enumerate} 
We denote the number of species with $m$, the number of complexes with $n$ and the number of reactions with $r$
\end{definition}

A CRN can be viewed as a digraph $(\mathscr{C},\mathscr{R})$ with vertex-labelling. In particular, it is a digraph where each vertex $y\in \mathscr{C}$ has positive degree and stoichiometry, i.e. there is a finite set $\mathscr{S}$ of species  such that $\mathscr{C}$ is a subset of $\mathbb{R}^{\mathscr{S}}_{\geq 0}$. The vertices are the complexes whose coordinates are in $\mathbb{R}^{\mathscr{S}}_{\geq 0}$, which  are the \textbf{stoichiometric coefficients}. The arcs are precisely the reactions.

We use the convention that an element $R_j = (y_j, y_j') \in \mathscr{R}$ is denoted by $R_j: y_j \rightarrow y_j' $. In this reaction, we say that $y_j$ is the \textbf{reactant} complex and $y'_j$ is the \textbf{product} complex. Connected components of a CRN are called \textbf{linkage classes}, strongly connected components are called \textbf{strong linkage classes}, and strongly connected components without outgoing arcs are called \textbf{terminal strong linkage classes}. We denote the number of linkage classes with $\ell$, that of the strong linkage classes with $s\ell$, and that of terminal strong linkage classes with $t$. A complex is called \textbf{terminal} if it belongs to a terminal strong linkage class; otherwise, the complex is called \textbf{nonterminal}. 

With each reaction $y\rightarrow y'$, we associate a \textbf{reaction vector} obtained by subtracting the reactant complex $y$ from the product complex $y'$. The \textbf{stoichiometric subspace} $S$ of a CRN is the linear subspace of $\mathbb{R}^\mathscr{S}$ defined by
$$S := \text{span }\{y' - y \in \mathbb{R}^\mathscr{S}| y\rightarrow y' \in \mathscr{R}\}.$$
The \textbf{rank} of the CRN, $s$, is defined as $s = \dim S$. 

Many features of CRNs can be examined by working in terms of finite dimensional spaces $\mathbb{R}^\mathscr{S}$ , $\mathbb{R}^\mathscr{C}$ , and $\mathbb{R}^\mathscr{R}$. Suppose the set $\{ \omega_i \in \mathbb{R}^\mathscr{I} \mid i \in \mathscr{I} \}$ forms the \textit{standard basis} for $\mathbb{R}^\mathscr{I}$ where $\mathscr{I}=\mathscr{S,C}$ or $\mathscr{R}$. We recall four maps relevant in the study of CRNs: map of complexes, incidence map, stoichiometric map and Laplacian map. 
\begin{definition}
Let $\mathscr{N}=(\mathscr{S,C,R})$ be a CRN. 
\begin{enumerate}
\item The \textbf{map of complexes} $\displaystyle{Y: \mathbb{R}^\mathscr{C} \rightarrow \mathbb{R}^\mathscr{S}}$ maps the basis vector $\omega_y$ to the complex $ y \in \mathscr{C}$. 
\item The \textbf{incidence map} $\displaystyle{I_a : \mathbb{R}^\mathscr{R} \rightarrow \mathbb{R}^\mathscr{C}}$ is the linear map defined by mapping for each reaction $\displaystyle{R_j: y_j \rightarrow y_j' \in \mathscr{R}}$, the basis vector $\omega_j$ to the vector $\omega_{y_j'}-\omega_{y_j} \in \mathscr{C}$. 
\item The \textbf{stoichiometric map} $\displaystyle{N: \mathbb{R}^\mathscr{R} \rightarrow \mathbb{R}^\mathscr{S}}$ is defined as $N = Y  I_a$. 
\item For each $k \in \mathbb{R}^\mathscr{R}_{>0}$ , the linear transformation  $A_k : \mathbb{R}^\mathscr{C} \rightarrow \mathbb{R}^\mathscr{C}$ called \textbf{Laplacian map} is the mapping defined by 
$$A_k x:= \sum_{y \rightarrow y'\in\mathscr{R}}k_{y \rightarrow y'}x_y (\omega_{y'} -\omega_y),$$
where $x_y$ refers to the $y^\text{th}$ component of $x \in \mathbb{R}^\mathscr{C}$ relative to the standard basis. 
\end{enumerate}
\end{definition}

A non-negative integer, called the deficiency, can be associated to each CRN. The \textbf{deficiency} of a CRN, denoted by $\delta$, is the integer defined by $\delta = n - \ell - s$. This structural index has been the center of many studies in CRNT due to its relevance in the dynamic behavior of the system. 

\subsection{Dynamics of chemical reaction networks}
By \textit{kinetics} of a CRN, we mean the assignment of a rate function to each reaction in the CRN. It is defined formally as follows.
\begin{definition}
A \textbf{kinetics} of a CRN $\mathscr{N}=(\mathscr{S},\mathscr{C},\mathscr{R})$ is an assignment of a rate function $\displaystyle{K_{j}: \Omega_K \to \mathbb{R}_{\geq 0}}$ to each reaction $R_j \in \mathscr{R}$, where $\Omega_K$ is a set such that $\mathbb{R}^{\mathscr{S}}_{>0} \subseteq \Omega_K \subseteq {\mathbb{R}}^{\mathscr{S}}_{\geq 0}$. A kinetics for a network $\mathscr{N}$ is denoted by $$\displaystyle{K=[K_1,K_2,...,K_r]^\top:\Omega_K \to {\mathbb{R}}^{\mathscr{R}}_{\geq 0}}.$$ The pair $(\mathscr{N},K)$ is called the \textbf{chemical kinetic system (CKS)}.
\end{definition}

The above definition is adopted from \cite{WIUF2013}. It is expressed in a more general context than those typically found in CRNT literature. For power law kinetic systems, one sets $\Omega_K =\mathbb{R}^{\mathscr{S}}_{>0}$. Here, we focus on the kind of kinetics relevant to our context: 

\begin{definition}
A \textbf{chemical kinetics} is a kinetics $K$ satisfying the positivity condition: 
$$
\text{For each reaction } R_j: y_j \rightarrow y_j' \in \mathscr{R},  \text{ } K_{j}(c)>0 \text{ if and only if } \text{supp } y_j \subset\text{supp }c.
$$
\end{definition}

\noindent Once a kinetics is associated with a CRN, we can determine the rate at which the concentration of each species evolves at composition $c \in \mathbb{R}^\mathscr{S}_{>0}$. 

\begin{definition}\label{def:SFRF}
The \textbf{species formation rate function}  of a chemical kinetic system is the vector field  
$$f(c) = NK (c) = \displaystyle\sum_{y_j\rightarrow y'_j \in \mathscr{R}}K_j(c) (y_j'- y_j).$$
\noindent The equation $dc/dt=f(c)$ is the \textbf{ODE or dynamical system} of the CKS.  A \textbf{positive equilibrium or steady state} $c^*$ is an element of $\mathbb{R}^\mathscr{S}_{>0}$ for which $f(c^*) = 0$. The set of positive equilibria of a chemical kinetic system is denoted by $\bm{E_+(\mathscr{N}, K)}$. 
\end{definition}

The \textit{complex formation rate function} is the analogue of the species formation rate function for complexes. 
 
\begin{definition}\label{def:CFRF}
The \textbf{complex formation rate function}  $g: \mathbb{R}^\mathscr{S}_{>0} \rightarrow \mathbb{R}^\mathscr{C}$ of a chemical kinetic system is the given by
\begin{equation}\label{eq:CFRF}
g(c) = I_a K (c) = \displaystyle\sum_{y_j\rightarrow y'_j \in \mathscr{R}}K_j(c) (\omega_{y_j'}- \omega_{y_j}).
\end{equation} 
where $I_a$ is the incidence map.
\end{definition}

Horn and Jackson \cite{HORNJACK1972} introduced the notion of \textit{complex balancing} in chemical kinetics, which proved to have profound uses in CRNT.  This is the counterpart of a positive steady state in the complex space, i.e. a concentration $c \in \mathbb{R}^\mathscr{S}_{>0}$ such that $g(c)=0$. It has a natural interpretation: Observe from Equation (\ref{eq:CFRF}) that the function $g$ gives the difference between the production and degradation of each complex. Thus, ``complex balancing'' occurs when $g(c)=0$.  In view of Definitions \ref{def:SFRF} and \ref{def:CFRF}, it is clear that
$$
f(c)=Yg(c).
$$
Hence, if $c \in \mathbb{R}^\mathscr{S}_{>0}$ is complex balanced, then $c$ is a steady state (as the linearity of $Y$ implies $Y(0)=0$). However, the converse does not necessarily hold (i.e., when $\text{Ker } Y$ is nontrivial).

\begin{definition}
A chemical kinetic system $(\mathscr{N},K)$ is called \textbf{complex balanced} if it has a complex balanced steady state. The set of positive complex balanced steady states of the system is denoted by $\bm{Z_+(\mathscr{N},K)}$. 
\end{definition}

We recall the following well-known result that establishes the relationship between weak reversibility and existence of complex balanced equilibria:

\begin{proposition}[Horn, \cite{HORN1972}]
If a chemical kinetic system has a complex balanced equilibrium, then the underlying CRN is weakly reversible. 
\end{proposition}

\subsection{Power law kinetic system}

Power law kinetics is defined by an  $r \times m$ matrix $F=[F_{ij}]$, called the \textbf{kinetic order matrix}, and vector $k \in \mathbb{R}^\mathscr{R}_{>0}$, called the \textbf{rate vector}.  

\begin{definition}
A kinetics $K: \mathbb{R}^\mathscr{S}_{>0} \rightarrow \mathbb{R}^\mathscr{R}$ is a \textbf{power law kinetics} (PLK) if
$$\displaystyle K_{i}(x)=k_i x^{F_{i,\cdot}} \quad \text{for all } i \in \overline{1,r}.$$
with $k_i \in \mathbb{R}_{>0}$ and $F_{ij} \in \mathbb{R}$.  A PLK system has \textbf{reactant-determined kinetics} (of type \textbf{PL-RDK}) if for any two reactions $R_i$, $R_j \in \mathscr{R}$ with identical reactant complexes, the corresponding rows of kinetic orders in $F$ are identical, i.e. $F_{ih}=F_{jh}$ for $h  \in \overline{1,m}$.  On the other hand, a PLK system has \textbf{non-reactant-determined kinetics} (of type \textbf{PL-NDK}) if there exist two reactions with the same reactant complexes whose corresponding rows of kinetic orders in $F$ are not identical. 
\end{definition}

An example of PL-RDK is the well-known \textbf{mass action kinetics} (MAK), where the kinetic order matrix is the transpose of the matrix representation of the map of complexes $Y$ \cite{FEIN1979}. That is, a kinetics is a MAK if
$$
K_{j}(x)=k_{j}x^{Y_{.,j}} \quad \text{for all } R_j: y_j \rightarrow y'_j \in \mathscr{R}
$$
where $k_{j} \in \mathbb{R}_{>0}$, called rate constants. Note that $Y_{.,j}$ pertains to the stoichiometric coefficients of a reactant complex $y_j \in \mathscr{C}$. 

\begin{remark}
In \cite{AJMSM2015}, Arceo et al. discussed several sets of kinetics of a network and drew a ``kinetic landscape". They identified two main sets: the \textbf{complex factorizable} (CF) kinetics and its complement, the \textbf{non-complex factorizable} (NF) kinetics. Complex factorizable kinetics generalize the key structural property of MAK -- that is, the species formation rate function decomposes as 
$$
\dfrac{dx}{dt}= Y \circ A_k \circ \Psi_k,
$$
where $Y$ is the map of complexes, $A_k$ is the Laplacian map, and $\Psi_k: \mathbb{R}^\mathscr{S}_{\geq 0} \rightarrow  \mathbb{R}^\mathscr{C}_{\geq 0}$ such that $I_a \circ K(x) = A_k \circ \Psi_k(x)$ for all $x \in \mathbb{R}^\mathscr{S}_{\geq 0}$. In the set of power law kinetics, the complex-factorizable kinetic systems are precisely the PL-RDK systems. 
\end{remark}

\subsection{Dynamical Equivalence of CRNs}

Two distinct CRNs with the same set of kinetics may give rise to identical set of ordinary differential equations. Such systems are said to be \textbf{dynamically equivalent}. This idea had been tackled as early as 1970s. For instance, Horn and Jackson \cite{HORNJACK1972} studied dynamical equivalence (which they termed as \textit{macro-equivalence)} for a class of weakly reversible MAK systems.  An extensive study of the dynamical equivalence of MAK systems was done by Craciun and Pantea \cite{CP2008}. 

The idea of dynamic equivalence is useful in understanding the qualitative behavior of chemical kinetic systems. If a kinetic system is found to be dynamically equivalent to another system that possesses desirable features about its dynamics (e.g. existence of positive steady state, capacity for multiple steady states, etc.) or network structure (e.g. weak reversibility, low deficiency, etc.), then the dynamical property of the desirable system applies for the system that does not have the desirable features. 

\subsection{Absolute concentration robustness (ACR)}

Formally, ACR is defined as follows.

\begin{definition}\label{def:acr}
A PL-RDK system $(\mathscr{N},K)$ has \textbf{absolute concentration robustness} (ACR) in a species $X \in \mathscr{S}$ if there exists $c^*\in E_+(\mathscr{N},K)$ and for every other $c^{**} \in E_+(\mathscr{N},K)$,  we have $c^{**}_X =c^*_{X}$.
\end{definition}

Shinar and Feinberg \cite{SF2010} established simple sufficient conditions for a MAK system to exhibit ACR. In \cite{FLRM2020}, it was shown that this result can be extended to deficiency one PL-RDK systems.  The extension of the Shinar-Feinberg Theorem on ACR for PL-RDK systems is stated below.

\begin{theorem}[Shinar-Feinberg Theorem on ACR for PL-RDK systems, \cite{FLRM2020}] \label{th:SFTACR}
Let $\mathscr{N}=(\mathscr{S,C,R})$ be a deficiency-one CRN and suppose that $(\mathscr{N},K)$ is a PL-RDK system which admits a positive equilibrium.  If $y, y' \in \mathscr{C}$ are nonterminal complexes whose kinetic order vectors  differ only in species $X$, then the system has ACR in $X$.
\end{theorem}

\section{ACR in deficiency zero PL-RDK and minimally PL-NDK systems} 

For convenience, we introduce the following terminology to refer to a pair of reactions whose reactants' kinetic order vectors differ only in one species. 

\begin{definition}
A pair of reactions in a PLK system is called a \textbf{Shinar-Feinberg pair} (or \textbf{SF-pair)} in a species $X$ if their kinetic order vectors differ only in $X$. A subnetwork of the PLK system is of \textbf{SF-type} if  it contains an SF-pair in $X$.
\end{definition}

We present an ACR theorem for deficiency zero PL-RDK systems and for a class of deficiency zero PL-NDK systems. We denote the later as ``minimally PL-NDK" because in terms of their NDK properties, they take minimal values: a single NDK node, two complex factorizable (CF-)subsets and in the special case of binary nodes, a single reaction in each CF-subset. For both PLK systems, the key property for ACR in a species $X$ is the presence of an SF-reaction pair. We use the CF-RM$_+$ method introduced in \cite{NEML2019} to show its dynamic equivalence with an appropriate deficiency one PL-RDK system. We provide examples to illustrate the approach.

\begin{definition}
A PL-NDK system is \textbf{minimally PL-NDK} if it contains a single NDK node which has two complex factorizable subsets (CF-subsets), at least one of which contains only one reaction. Such a node is called a \textbf{minimal NDK node}. If both CF-subsets have only one reaction, we call the node a \textbf{binary NDK node}.
\end{definition}

The \textbf{CF-RM}$\bm{_+}$ method, an algorithm introduced in \cite{NEML2019}, transforms a PL-NDK system to a dynamically equivalent PL-RDK system. The procedure is as follows: at each NDK node, except for a CF-subset with a maximal number of reactions, the reactions in a CF-subset are replaced by adding the same reactant multiple to reactant and product complexes, such that the new reactants and products do not coincide with any existing complexes. Suppose $(\mathscr{N},K)$ is a PL-NDK system that is transformed into a PL-RDK system $(\mathscr{N}^*,K^*)$ via CF-RM$_+$ algorithm. The two key properties of $\mathscr{N}$ and $\mathscr{N}^*$ are the invariance of the stoichiometric subspaces, i.e. $S=S^*$, and the kinetic order matrices, $F=F^*$. Details of the algorithm can be found in \cite{NEML2019}.

\begin{theorem} \label{th:acrdz}
Let $(\mathscr{N},K)$ be a deficiency zero PL-RDK or minimally PL-NDK system with a positive equilibrium. If the system is of SF-type in a species $X$, then it has ACR in $X$.
\end{theorem}

\begin{proof}
Note that after the classical results of Feinberg and Horn, any positive equilibrium in a deficiency zero system is complex balanced \cite{FEIN1972} and hence, the underlying CRN is weakly reversible \cite{HORN1972}. We begin with the PL-NDK case.

Let $y \rightarrow y'$ be single reaction in the hypothesized CF-subset of the minimal NDK node. Applying $\text{CF-RM}_+$ method  to transform $(\mathscr{N},K)$, we obtain as a transform of $\mathscr{N}$ the network $\mathscr{N}^*$ with $\mathscr{S}^*=\mathscr{S}$, $\mathscr{C}^* = \mathscr{C} \cup \{y + ay, y'+ay \}$ where $a$ is an appropriate integral multiple of $y$ and $\mathscr{R}^*=\mathscr{R} \cup \{ y + ay \rightarrow y' + ay \}$. Since we assume that the network has a complex balanced equilibrium, then by a classical result of Horn \cite{HORN1972} , it is weakly reversible, and hence each linkage class is weakly reversible. Since each reaction in the linkage class of the NDK node is in a cycle, the CF-RM$_+$ creates only one additional linkage class, namely $\{ y + ay \rightarrow y' + ay \}$. Hence, the deficiency of $\mathscr{N}^*$ is $\delta^* = (n+2)-(\ell+1)-s=\delta+1=1$. The kinetic order matrix remains the same, so the transform $\mathscr{N}^*$ is still of SF-type in $X$. $\mathscr{N}^*$, as a dynamically equivalent PL-RDK system, has a positive equilibrium and hence, fulfill the assumptions of the extension of the Shinar-Feinberg ACR Theorem for PL-RDK system (Theorem \ref{th:SFTACR}). Hence, $\mathscr{N}^*$ has ACR in $X$. 

In the PL-RDK case, we can apply the CF-RM$_+$ method to any reaction and also obtain an appropriate dynamically equivalent deficiency one system as in the minimally PL-NDK case. 
\end{proof}

\noindent Note that an adaptation of the direct proof in \cite{FLRM2020} to the deficiency zero PL-RDK system is provided in the Appendix. However, the argument yields only a restricted result. 

\begin{corollary}\label{cor:monospecies}
Let $(\mathscr{N},K)$ be a deficiency zero, minimally PL-NDK system with a complex balanced equilibrium. Suppose the reactant of the NDK node is monospecies, i.e. it is of the form $nX$ for some positive integer $n$ and species $X$. Then $(\mathscr{N},K)$ has ACR in $X$.
\end{corollary}

\begin{proof}
Since the node is monospecies, the kinetic order vectors of its branching reactions have non-zero values only in $X$. Since it is an NDK, those non-zero values must be different. Hence, a reaction from one CF-subset and one from the other form an SF-pair, and the claim follows from the previous proposition.
\end{proof}

\begin{example}\label{ex:zero-ANPRI}
In \cite{FLRM2020}, it was shown that the conditions of the Shinar-Feinberg Theorem on ACR for PL-RDK systems are satisfied by a PL-RDK system representation for a power law approximation of the pre-industrial carbon cycle model of Anderies et al. \cite{ANDERIES} , and thus it exhibits ACR in a species. Here, we consider its dynamically equivalent deficiency zero PL-RDK system with associated kinetic order matrix $F$:

\begin{equation}
\left.
\arraycolsep=1.4pt\def\arraystretch{2.5}
  \begin{array}{rcl}
A_1 + 2A_2 &\overset{R_1}{\underset{R_2}\rightleftarrows} &2A_1 + A_2 \\
A_2 &\overset{R_3}{\underset{R_4}\rightleftarrows}  & A_3 \\
  \end{array}
 \right.
 \quad \quad
 {\small
F=  
\kbordermatrix{
    & A_1 & A_2 & A_3  \\
    R_1 & p_1 & q_1 & 0   \\
    R_2 & p_2 & q_2 & 0   \\
    R_3 & 0 & 1 & 0   \\
    R_4 & 0 & 0 & 1   \\
},}
\end{equation}
where $p_1=p_2=-68$ and $q_1=0.58$, and $q_2=0.91$. Since $\{R_1 , R_2 \}$ is an SF-pair in $A_2$, it follows from Theorem \ref{th:acrdz}  that there is ACR in $A_2$. 
\end{example}

\begin{example} \label{ex:zero-IDHKP}
The Shinar-Feinberg theorem on ACR for MAK systems \cite{SF2010} provided theoretical support to empirically observed concentration robustness inisocitrate dehydrogenase kinase-phosphatase-isocitrate dehydrogenase (IDHKP-IDH)  glyoxylate bypass control system. Using the technique of network translation of Johnston \cite{JOHNSTON2014}, the MAK system of  IDHKP-IDH  glyoxylate bypass control system has the following dynamically equivalent weakly reversible deficiency zero PL-RDK system.
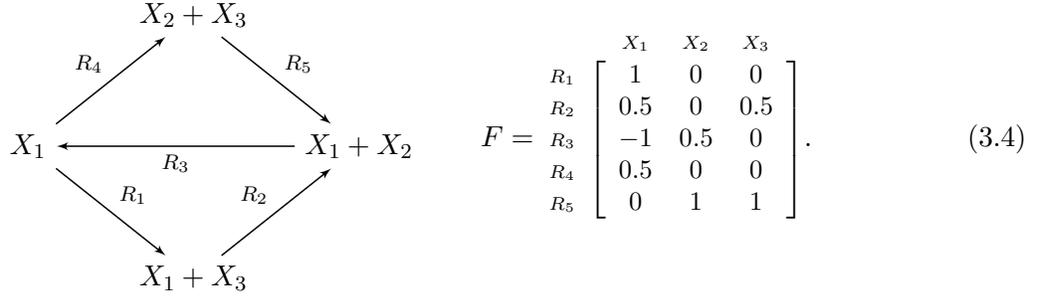
\begin{equation}
\begin{tikzpicture}[baseline=(current  bounding  box.center)]
\tikzset{vertex/.style = {draw=none,fill=none}}
\tikzset{edge/.style = {->,> = latex', line width=0.15mm}}
\node[vertex] (1) at  (0,0) {$EI_p+I+E$};
\node[vertex] (2) at  (3.5,0) {$EI_pI+E$};
\node[vertex] (3) at  (3.5,-1.5) {$EI_p+I_p+E$};
\node[vertex] (4) at  (0,-1.5) {$2EI_p$};
\scriptsize
\draw [edge]  (1.25,0.2) to["$R_1$"] (2.5,0.2);
\draw [edge]  (2) to["$R_2$"] (1);
\draw [edge]  (2) to["$R_3$"] (3);
\draw [edge]  (2.2,-1.55) to["$R_4$"](0.6,-1.55) ;
\draw [edge]  (0.6,-1.35) to["$R_5$"] (2.2,-1.35);
\draw [edge]  (4) to["$R_6$"] (1);
\end{tikzpicture}
\quad
 {\small
F=  
\kbordermatrix{
    & EI_p & I & EI_pI & I_p & E  \\
    R_1 & 1 & 1 & 0 & 0 & 0   \\
    R_2 & 0 & 0 & 1 & 0 & 0   \\
    R_3 & 0 & 0 & 1 & 0& 0   \\
    R_4 & 0 & 0 & 0 & 1 & 1 \\
    R_5 & 1 & 0 & 0 & 0 & 0 \\
    R_6 & 1 & 0 & 0 & 0 & 0 \\
}}
\end{equation}
$\{ R_1, R_5 \}$ forms an SF-pair in $I$ and hence, it follows from Theorem \ref{th:acrdz} that the system has ACR in species $I$, which agrees with the result in \cite{SF2010}.
\end{example}

\begin{remark}
Under mass action kinetics, any deficiency zero and conservative (i.e. the orthogonal complement of its stoichiometric subspace meets $\mathbb{R}^\mathscr{S}_{>0}$) CRN cannot exhibits absolute concentration robustness \cite{SF2011}. These properties that thwart ACR for MAK systems, however, do not extend to power law kinetics as shown in Examples \ref{ex:zero-ANPRI} and \ref{ex:zero-IDHKP}. These two deficiency zero and conservative PL-RDK systems display ACR.   
\end{remark}

\begin{example}
A deficiency zero subnetwork of Schmitz's pre-industrial carbon cycle model \cite{SCHM2002}  studied in Fortun et al. \cite{FMRL2019} is shown below. Its kinetic order matrix $F$ is also provided.
\begin{equation}\label{fig:Example1}
\begin{tikzpicture}[baseline=(current  bounding  box.center)]
\tikzset{vertex/.style = {draw=none,fill=none}}
\tikzset{edge/.style = {->,> = latex', line width=0.15mm}}
\node[vertex] (1) at  (0,0) {$M_1$};
\node[vertex] (2) at  (1.75,1.5) {$M_2$};
\node[vertex] (3) at  (1.75,-1.5) {$M_3$};
\node[vertex] (4) at  (3.5,0) {$M_4$};
\node[vertex] (5) at  (-1.75,1.5) {$M_5$};
\node[vertex] (6) at  (-1.75,-1.5) {$M_6$};
\scriptsize
\draw [edge]  (5.340) to["$R_1$"] (1.130);
\draw [edge]  (1.155) to["$R_2$"] (5.315);
\draw [edge]  (5) to["$R_3$"] (6);
\draw [edge]  (6) to["$R_4$"] (1);
\draw [edge]  (1.340) to["$R_8$"] (3.130);
\draw [edge]  (4.155) to["$R_6$"] (2.315);
\draw [edge]  (2.225) to["$R_5$"] (1.25);
\draw [edge]  (3.50) to["$R_7$"] (4.200);
\end{tikzpicture}
\quad \quad 
{\footnotesize
F=  
\kbordermatrix{
    & M_1 & M_2 & M_3 & M_4 & M_5 & M_6  \\
    R_1 & 0 & 0 & 0 & 0 & 1 & 0   \\
    R_2 & 0.36 & 0 & 0 & 0 & 0 & 0   \\
    R_3 & 0 & 0 & 0 & 0 & 1 & 0   \\
    R_4 & 0 & 0 & 0 & 0 & 0 & 1   \\
    R_5 & 0 & 9.4 & 0 & 0 & 0 & 0  \\
    R_6 & 0 & 0 & 0 & 1 & 0 & 0  \\
    R_7 & 0 & 0 & 1 & 0 & 0 & 0  \\
    R_8 & 1 & 0 & 0 & 0 & 0 & 0  \\
}.}
\end{equation}
\noindent It is a minimally PL-NDK system with a complex balanced equilibrium. The reaction pair $\{ R_2 , R_8 \}$ form an SF-pair in $M_1$. Since the reactant complex $M_1$ of the single binary NDK node is monospecies, it follows from Corollary \ref{cor:monospecies} that it has ACR in $M_1$.
\end{example}

\begin{example}
Consider the following network with species set $\{X_1, X_2, X_3, X_4 \}$ and power law kinetics given by the kinetic order matrix $F$.

\begin{equation}
\begin{tikzpicture}[baseline=(current  bounding  box.center)]
\tikzset{vertex/.style = {draw=none,fill=none}}
\tikzset{edge/.style = {->,> = latex', line width=0.20mm}}
\node[vertex] (1) at  (0,0) {$X_2+X_3$};
\node[vertex] (2) at  (0,-3.5) {$X_1+X_3$};
\node[vertex] (3) at  (2.2,-1.75) {$X_1+X_2$};
\node[vertex] (4) at  (-2.2,-1.75) {$X_1$};
\scriptsize
\draw [edge]  (4) to["$R_1$"] (2);
\draw [edge]  (2) to["$R_2$"] (3);
\draw [edge]  (1) to["$R_5$"] (3);
\draw [edge]  (4) to["$R_4$"] (1);
\draw [edge]  (3) to["$R_3$"] (4);
\end{tikzpicture}
\quad \quad 
{\small
F=  
\kbordermatrix{
    & X_1 & X_2 & X_3  \\
    R_1 & 1 & 0 & 0   \\
    R_2 & 0.5 & 0 & 0.5   \\
    R_3 & -1 & 0.5 & 0   \\
    R_4 & 0.5 & 0 & 0   \\
    R_5 & 0 & 1 & 1   \\
}.}
\end{equation}
$X_1$ is the only NDK node, and it is binary. The reaction pairs $\{R_1, R_4 \}$ and $\{ R_2, R_4\}$ are SF pairs in $X_1$ and $X_3$, respectively. The stoichiometric matrix $N$ of the network is given by:
$$ 
N=  
\kbordermatrix{
    & R_1 & R_2 & R_3 & R_4 & R_5 \\
    X_1 & 0 & 0 & 0 & -1 & 1  \\
    X_2 & 0 & 1 & -1 & 1 & 0   \\
    X_3 & 1 & -1 & 0 & 1 & -1   \\
}.
$$

\noindent Since the rows of $N$ are linearly independent, $s=3$, and hence there is only one stoichiometric class. The deficiency $\delta = 4-1-3=0$. The ODE system is the following:
\begin{align*}
\frac{dX_1}{dt} & = k_5 X_2 X_3 - k_4 X_1^{0.5} \\
\frac{dX_2}{dt} & = k_2 X_1^{0.5} X_3 ^{0.5} + k_4 X_1^{0.5} - k_3 X_1^{-1} X_2^{0.5} \\
\frac{dX_3}{dt} &= k_1 X_1 + k_4 X_1^{0.5} - k_2 X_1^{0.5} X_3^{0.5} - k_5 X_2 X_3
\end{align*}

\noindent For the rate vector $k = (1,1,2,1,1)$, the system has the steady state $(1,1,1)$. Hence, the system has ACR in $X_1$ and $X_3$. In general, for rate vectors satisfying the equations $k_1 = k_2$  and $k_3 = (k_1+k_4)(\frac{k_5}{k_4})^{0.5}$, the equilibrium is given by $(1,\frac{k_4}{k_5},1)$. 

\end{example}

\section{Decomposition theory and ACR}

In this Section, we use decomposition theory to identify large classes of PLK systems, including such with higher deficiency, i.e. $\delta \geq 2$.  These results suggest that ACR is essentially a ``local'' property of a low deficiency subnetwork, which serves as a ``building block.'' We first review the required concepts and results from decomposition theory and then formulate the new results on ACR.

\subsection{A review of decomposition theory}

We refer to \cite{FML2020} for more details on the concepts and results in decomposition theory.

\begin{definition}
Let $\mathscr{N} =(\mathscr{S}, \mathscr{C}, \mathscr{R})$ be a CRN. A \textbf{covering} of $\mathscr{N}$ is a collection of subsets $\{ \mathscr{R}_1, \mathscr{R}_2,\dots, \mathscr{R}_p \}$ whose union is $\mathscr{R}$. A covering is called a \textbf{decomposition} of $\mathscr{N}$ if the sets $\mathscr{R}_i$ form a partition of $\mathscr{R}$.
\end{definition}

Clearly, each $\mathscr{R}_i$ defines a subnetwork $\mathscr{N}_i$ of $\mathscr{N}$, namely $\mathscr{C}_i$ consisting of all complexes occurring in $\mathscr{R}_i$ and $\mathscr{S}_i$ consisting of all the species occurring in $\mathscr{C}_i$.

\begin{proposition}[Prop. 3., \cite{FML2020}] 
If $\{ \mathscr{R}_1, \mathscr{R}_2,\dots, \mathscr{R}_p \}$ is a network covering, then
\begin{enumerate}[(i)]
\item $S= S_1 + S_2 + \cdots + S_p$;
\item $s \leq s_1 + s_2 + \cdots + s_p$, where $s=\dim S$ and $s_i=\dim S_i$ for $i\in \overline{1,p}$.
\end{enumerate}
\end{proposition}

Feinberg  \cite{FEIN1987} identified the important subclass of independent decomposition:

\begin{definition}
A decomposition is \textbf{independent} if the $S$ is the direct sum of the subnetworks' stoichiometric subspaces $S_i$ or equivalently, if $s = s_1 + s_2 + \cdots + s_p$.
\end{definition}

In \cite{FMRL2019}, Fortun et al. derived a basic property of independent decompositions:

\begin{proposition}
If $\mathscr{N}=\mathscr{N}_1 \cup \mathscr{N}_2 \cup \cdots  \cup \mathscr{N}_p$ is an independent decomposition, then $
\delta \leq\delta_1 +\delta_2 + \cdots +\delta_p$, where $\delta_i$ represents the deficiency of the subnetwork $\mathscr{N}_i$. 
\end{proposition}

Feinberg \cite{FEIN1987} established the relationship between the positive equilibria of the ``parent network'' and those of the subnetworks of an independent decomposition:

\begin{theorem}[Feinberg Decomposition Theorem, \cite{FEIN1987}] \label{feinberg theorem}
Let $\{\mathscr{R}_1,\mathscr{R}_2,\dots, \mathscr{R}_p \}$ be a partition of a CRN $\mathscr{N}$ and let $K$ be a kinetics on $\mathscr{N}$. If $\mathscr{N}=\mathscr{N}_1 \cup \mathscr{N}_2 \cup \cdots \cup\mathscr{N}_p$ is the network decomposition generated by the partition  and $E_+(\mathscr{N}_i,K_i)= \{ x \in \mathbb{R}^\mathscr{S}_{>0} | N_i K_i(x) = 0 \}$, then \begin{enumerate}
\item[(i)] $E_+ (\mathscr{N}, K) \subseteq \displaystyle{\bigcap_{i\in \overline{1,p}}} E_+ (\mathscr{N}_i, K_i)$
\item[(ii)] If the network decomposition is independent, then equality holds.
\end{enumerate}
\end{theorem}

Farinas et al. \cite{FML2020}  introduced the concept of an \textit{incidence independent decomposition}, which naturally complements the independence property. Our starting point is the following basic observation:

\begin{proposition}[Prop. 6, \cite{FML2020}]
If $\{ \mathscr{R}_i \}$ is a network covering, then
\begin{enumerate}[(i)]
\item $\text{\emph{Im} } I_a = \text{\emph{Im }} I_{a,1} + \text{\emph{Im }} I_{a,2} + \cdots + \text{\emph{Im }} I_{a,p}$, where $I_{a,i}$ denotes the incidence map of the subnetwork $\mathscr{N}_i$.
\item $n- \ell \leq (n_1 - \ell_1) + (n_2 - \ell_2) + \cdots + (n_p - \ell_p)$, where $n-\ell=\dim I_a$ and $n_i - \ell_i=\dim I_{a,i}$ for $i\in \overline{1,p}$.
\end{enumerate}
\end{proposition}

The analogous concept to independent decomposition is the following:

\begin{definition}
A decomposition $\{ \mathscr{N}_1, \mathscr{N}_2, \dots, \mathscr{N}_p \}$ of a CRN is \textbf{incidence independent} if and only if the image of the incidence map $I_a$ of $\mathscr{N}$ is the direct sum of the images of the incidence maps of the subnetworks.
\end{definition}

It follows from this definition that the dimension of the image of the incidence map $I_a$ equals the sum of the dimensions of the subnetworks' incidence maps. That is, $n-\ell = \sum (n_i - \ell_i)$. 

\begin{example}
The linkage classes form the primary example of an incidence independent decomposition, since $n = \sum n_i$ and $\ell = \sum \ell_i$. In fact, the linkage class decompositions belong to the important subclass of $\mathscr{C}$-decompositions discussed Definition \ref{def:Cdecomp}.
\end{example} 
The following result is the analogue of the result of Fortun et al. \cite{FMRL2019} for incidence independent decomposition.

\begin{proposition}[Prop. 7, \cite{FML2020}]
\label{prop:incidenceindep}
Let $\mathscr{N}=\mathscr{N}_1 \cup \mathscr{N}_2 \cup \cdots \cup \mathscr{N}_p$ be an incidence independent decomposition. Then $\delta \geq \delta_1 +\delta_2 + \cdots + \delta_p$.
\end{proposition}

\begin{definition}
A decomposition is \textbf{bi-independent} if it is both independent and incidence independent.
\end{definition}
Independent linkage class decomposition is the best known example of bi-independent decomposition. 

\begin{proposition}[Prop. 9, \cite{FML2020}]
A decomposition $\mathscr{N}= \mathscr{N}_1 \cup \mathscr{N}_2 \cup \cdots \cup \mathscr{N}_p$ is independent or incidence independent and $\displaystyle{\sum_{i =1}^p} \delta_i = \delta$ if and only if $\mathscr{N}= \mathscr{N}_1 \cup \mathscr{N}_2 \cup \cdots \cup \mathscr{N}_p$ is bi-independent.
\end{proposition}
\noindent Note that for a deficiency zero network, an independent decomposition is incidence independent and therefore, bi-independent. 

$\mathscr{C}$-decompositions form an important class of incidence independent decompositions:

\begin{definition}\label{def:Cdecomp}
A decomposition $\mathscr{N} = \mathscr{N}_1 \cup \mathscr{N}_2 \cup \cdots \cup \mathscr{N}_p$ with $\mathscr{N}_i =(\mathscr{S}_i, \mathscr{C}_i, \mathscr{R}_i)$ is a \textbf{$\mathscr{C}$-decomposition} if $\mathscr{C}_i \cap \mathscr{C}_j = \emptyset$ for $i \neq j$.
\end{definition}

A $\mathscr{C}$-decomposition partitions not only the set of reactions but also the set of complexes. The primary examples of  $\mathscr{C}$-decomposition are the linkage classes. Linkage classes, in fact, essentially determine the structure of a $\mathscr{C}$-decomposition. 

\begin{theorem}[Structure Theorem for $\mathscr{C}$-decomposition, \cite{FML2020}]
Let $\mathscr{L}_1, \mathscr{L}_2, \dots, \mathscr{L}_\ell$ be the linkage classes of a network $\mathscr{N}$. A decomposition $\mathscr{N}=\mathscr{N}_1 \cup \mathscr{N} \cup \cdots \cup \mathscr{N}_p$ is a $\mathscr{C}$-decomposition if and only if each $\mathscr{N}_i$ is the union of linkage classes and each linkage class is contained in only one $\mathscr{N}_i$. In other words, the linkage class decomposition is a refinement of $\mathscr{N}$.
\end{theorem}

The following Theorem from \cite{FML2020} shows the relationship between the set of incidence independent decompositions and the set of complex balanced equilibria of any kinetic system.  It is the precise analogue of Feinberg's work (Theorem \ref{feinberg theorem}). 

\begin{theorem}[Theorem 4, \cite{FML2020}]
\label{th:Z}
Let $\mathscr{N}=(\mathscr{S}, \mathscr{C}, \mathscr{R})$ be a CRN and $\mathscr{N}_i =(\mathscr{S}_i, \mathscr{C}_i, \mathscr{R}_i)$ for $i\in \overline{1,p}$ be the subnetworks of a decomposition. Let $K$ be any kinetics, and $Z_+ (\mathscr{N},K)$ and $Z_+ (\mathscr{N}_i, K_i)$  be the set of  complex balanced equilibria of $\mathscr{N}$ and $\mathscr{N}_i$, respectively. Then
\begin{enumerate} 
\item[(i)] $\displaystyle{\bigcap_{i\in \overline{1,p}}} Z_+ (\mathscr{N}_i, K_i) \subseteq Z_+ (\mathscr{N}, K)$
\end{enumerate}
 If the decomposition is incidence independent, then 
 \begin{enumerate}
\item[(ii)] $Z_+ (\mathscr{N}, K)= \displaystyle{\bigcap_{i\in \overline{1,p}}} Z_+ (\mathscr{N}_i, K_i)$
\item[(iii)] $ Z_+ (\mathscr{N}, K) \neq \emptyset$ implies $Z_+ (\mathscr{N}_i, K_i) \neq \emptyset$ for each $i\in \overline{1,p}$.
\end{enumerate}

\end{theorem}

\subsection{ACR in PLK systems with a positive equilibrium}

We can now demonstrate ACR in classes of PL-NDK and higher deficiency PLK systems.

\begin{proposition}
Let $(\mathscr{N},K)$ be a PLK system with a positive equilibrium and an independent decomposition $\mathscr{N}= \mathscr{N}_1 \cup \mathscr{N}_2 \cup \cdots \cup \mathscr{N}_p$. If there is an $\mathscr{N}_i$ with $(\mathscr{N}_i,K_i)$ of SF-type in $X \in \mathscr{S}$ such that
\begin{enumerate}[(i)]
\item $\delta =0$ and is PL-RDK or minimally PL-NDK, or 
\item $\delta =1$ and is PL-RDK 
\end{enumerate}
Then $(\mathscr{N},K)$ has ACR in $X$.
\end{proposition}

\begin{proof}
Since $E_+(\mathscr{N},K)\neq \emptyset$ and the decomposition is independent, $E_+(\mathscr{N_i},K_i)\neq \emptyset$ for each $i \in \overline{1,p}$. We denote the PL-RDK or minimally subnetwork with $\mathscr{N}_\text{ACR}$ and associated kinetics $K_\text{ACR}$. The subnetwork $\mathscr{N}_\text{ACR}$ fulfills the conditions  for Theorem \ref{th:SFTACR} or Theorem \ref{th:acrdz} and hence, it has ACR in $X$ for all equilibria in $E_+(\mathscr{N}_\text{ACR},K_\text{ACR})$. Since the latter set contains $E_+(\mathscr{N},K)$, the PLK system $(\mathscr{N},K)$ has ACR in $X$.
\end{proof}

\noindent To ensure that higher deficiency network occur, we have:

\begin{corollary}
If the decomposition in the previous Proposition is bi-independent and at least one more subnetwork $\mathscr{N}_j$ has $\delta_j>1$, then the network $\mathscr{N}$ has higher deficiency.
\end{corollary}
\begin{proof}
For a bi-independent decomposition, we have $\delta = \delta_1 +\delta_2 + \cdots + \delta_p$.
\end{proof}

\subsection{BCR for classes of PLK systems with complex balanced equilibrium}

Incidence independent decompositions of CRNs are more common than independent ones. For instance, all linkage class decompositions are incidence independent, but few are independent. This motivates the introduction of a weaker form of concentration robustness than ACR:

\begin{definition}
A complex balanced chemical kinetic system $(\mathscr{N},K)$  has \textbf{balanced concentration robustness (BCR)} in a species $X \in \mathscr{S}$ if $X$ has the same value for all $c \in Z_+ (\mathscr{N},K)$.
\end{definition}

Clearly, a system that has ACR in a species implies that it also exhibits BCR for that species. A class of systems for which the converse holds (justifying the notation) is given by the following definition.

\begin{definition}
A complex balanced system is \textbf{absolutely complex balanced (ACB)} if $Z_+ (\mathscr{N},K)=E_+ (\mathscr{N},K)$.
\end{definition}

\noindent Examples of absolutely complex balanced systems are deficiency zero networks with positive equilibrium (for any kinetics) and complex balanced mass action systems (for any deficiency).

We have an analogous result for complex balanced systems and BCR:

\begin{proposition}
Let $(\mathscr{N},K)$ be a PLK system with a complex balanced equilibrium and an incidence independent decomposition $\mathscr{N}= \mathscr{N}_1 \cup \mathscr{N}_2 \cup \cdots \cup \mathscr{N}_p$. If there is an $\mathscr{N}_i$ with $(\mathscr{N}_i,K_i)$ of SF-type in $X \in \mathscr{S}$ such that
\begin{enumerate}[(i)]
\item $\delta =0$ and is PL-RDK or minimally PL-NDK, or 
\item $\delta =1$ and is PL-RDK 
\end{enumerate}
Then $(\mathscr{N},K)$ has BCR in $X$.
\end{proposition}

\begin{proof}
Since $Z_+(\mathscr{N},K)\neq \emptyset$ and the decomposition is independent, Theorem \ref{th:Z} guarantees that $Z_+(\mathscr{N_i},K_i)\neq \emptyset$ for each $i \in \overline{1,p}$. Denote the PL-RDK or minimally subnetwork with $\mathscr{N}_\text{ACR}$ and associated kinetics $K_\text{ACR}$. This subnetwork satisfies the conditions  for Theorem \ref{th:SFTACR} or Theorem \ref{th:acrdz} and hence, it has ACR in $X$ for all equilibria in $E_+(\mathscr{N}_\text{ACR},K_\text{ACR})$. Since the latter set contains $Z_+(\mathscr{N},K)$, the PLK system $(\mathscr{N},K)$ has BCR in $X$.
\end{proof}

\section{Discussion: the primarily kinetic character of ACR}

In this Section, we describe the evolution of the assessment of ACR as a system property -- from the emphasis on its ``structural sources'' in the original papers of Shinar and Feinberg \cite{SF2010,SF2011} in 2010 and 2011 to our current view of its primarily kinetic character.

Shinar and Feinberg entitled their groundbreaking paper \cite{SF2010} \textit{Structural sources of robustness in biochemical reaction networks} in which ``structural'' referred to the hypotheses (i) of the network's deficiency being equal to one and (ii) of the presence of two reactant complexes which differed only in a species $X$. The kinetic assumptions were the use of MAK and the existence of a positive equilibrium. In a further paper \cite{SF2011}, they emphasized these structural aspects by speaking of ``design principles'' for networks in order to achieve robustness.

The extension of the Shinar-Feinberg ACR Theorem for MAK systems by Fortun et al. \cite{FLRM2020} to their superset PL-RDK of power law kinetic systems with reactant-determined kinetics, maintained the first structural property but transformed the second to the kinetic property of a pair of reactions whose kinetic order vectors differed only in the coordinate for species $X$. In the special case of MAK systems, the kinetic order values coincide with the stoichiometric constants of the reactant complexes, thus ``hiding'' its kinetic character. Nevertheless, even in this extension, some structural aspects should be noted, as expressed in the following Proposition:

\begin{proposition}
Let $\{ R, R' \}$ be an SF-pair in  the species $X$ of a PLK system $(\mathscr{N},K)$ and $y, y'$ their reactant complexes. Then
\begin{enumerate}[(i)]
\item For any species $Y \neq X$, $Y \in \text{supp } y$ if and only if  $Y \in \text{supp } y'$.
\item $X \in \text{supp } y \cup \text{supp } y'$.
\item If the stoichiometric coefficients of complexes in $\mathscr{N}$ are only 0 or 1, then $y$ and $y'$ differ only in $X$.
\end{enumerate}
\end{proposition}

\noindent Examples for (iii) are total realizations of BST systems without self-regulating species (see \cite{FML2020} for details). 

In this paper, we highlighted the importance of the invariance of ACR under dynamic equivalence by identifying deficiency zero PLK systems through the use of the CF-RM$_+$ method. It is interesting to note that an adaptation of the direct proof in \cite{FLRM2020} to the deficiency zero case yields only a more restricted result (see Appendix). In our view, the usefulness of this invariance property further emphasizes the primarily kinetic character of ACR. A comparison with the property of a system of having a complex balanced equilibrium yields the complementary insight that the latter is often lost under dynamic equivalence, so that one could say, that the existence of a complex balanced equilibrium is a primarily structural property. 

With low deficiency subnetworks as ``building blocks'', results on decomposition help overcome the structural restriction initially suggested by the deficiency one requirement. Nevertheless, one should not forget that that special case is the starting point of the broader identification of ACR in PLK systems.

To conclude our discussion of this topic, we introduce the concept of a \textit{Birch system}.

\begin{definition}  \label{def:Birch}
A \textbf{Birch system} is a kinetic system with only one positive equilibrium (in the whole species space).
\end{definition}

Note that being a Birch system is a purely kinetic property. The name is derived from Birch's Theorem for weakly reversible deficiency zero MAK systems. If the system is open, there is only one stoichiometric class and the positive (in this case, complex balanced) equilibrium is unique in species space. 
\begin{example}
Any open PL-RDK system fulfilling the criterion of Craciun et al. \cite{CMPY2019}
\end{example}
\begin{example}
The independent realization or any open subnetwork realization of a regular S-system.
\end{example}

The following (straightforward) Proposition makes the connection between Birch systems and ACR explicit, revealing in our view its primarily kinetic character further:

\begin{proposition}
A chemical kinetic system is a Birch system if and only if it has ACR in every species.
\end{proposition}

\section{Summary and Outlook}

In conclusion, we summarize our results and outline some perspectives for further research.
\begin{enumerate}
\item We used the idea that ACR is a condition that is invariant under dynamic equivalence of CRNs to formulate new results that indicate ACR in deficiency zero PL-RDK system and minimally PL-NDK systems.  
The key concept needed in deriving these results involve the CF-RM$_+$ transformation of a system into a dynamically equivalent PLK system that fulfills the assumptions of the Shinar-Feinberg ACR Theorem for PL-RDK systems (Theorem \ref{th:SFTACR}). Examples were also provided to illustrate these results. Some of the examples, in fact,  showed that a standing result \cite{SF2011}, which establishes that deficiency zero precludes ACR for conservative mass action systems, does not hold for a more general kinetic system.
\item Using independent decomposition, we presented a result that identifies ACR in higher deficiency networks.  This result indicated that ACR is  a ``local'' property of a low deficiency subnetwork that serves as a ``building block'' of a larger ACR-possessing network.  
\item In a similar approach, we used incidence independent decomposition to investigate the dynamics of larger networks with low deficiency subnetworks that ACR in a species. This led us to identify a weaker concentration robustness than ACR, which we called `balanced concentration robustness' (BCR). A system is said to have BCR in a species if, for any complex balanced equilibrium that the system may admit, the value of that species remains the same.
Our result suggests that if a PLK system has complex balanced equilibrium, has incidence independent decomposition, and a subnetwork that has ACR for a species, then the system has BCR for that species.
\item In previous literature, much attention is given to the view that ACR is a property that is conferred from structural sources. Here, however, we provided a discussion that emphasized the primarily kinetic character of ACR. 
\item We plan to work on computational approaches and tools for identification of ACR and BCR using the results presented in this paper.
\end{enumerate}

\appendix
\section{Appendix: ACR in a deficiency zero PL-RDK systems}

In this Section, we provide an adaptation of the direct proof  \cite{FLRM2020} of Theorem \ref{th:SFTACR}  to the deficiency zero case. Unlike Theorem \ref{th:acrdz}, however, this result leads only to a restricted result (i.e. SF-pairs belong to the same linkage class).

\begin{theorem}\label{th:DZACR}
Let $\mathscr{N}$ be deficiency zero PL-RDK system that has a positive equilibrium. If a pair of reactions in a linkage class forms an SF-pair in species $X$, then the system has ACR in $X$.
\end{theorem}

The following result, named as the \textit{Structure Theorem of the Laplacian Kernel} (STLK) by Arceo et al. in \cite{AJMSM2015}, is crucial in proving the Theorem \ref{th:DZACR}. 

\begin{proposition}[Structure Theorem of the Laplacian Kernel (STLK)]\label{th: STLK}
Let $\mathscr{N}=\mathscr{(S,C,R)}$ be a CRN with terminal strong linkage classes $\mathscr{C}^1, \mathscr{C}^2, \dots, \mathscr{C}^t$. Let $k\in \mathbb{R}^\mathscr{R}_{>0}$ and $A_k$ its associated Laplacian. Then $\text{Ker } A_k$ has a basis $b^1,b^2,\dots,b^t$ such that $\text{supp }b^i = \mathscr{C}^i$ for all $i\in \overline{1,t}$.
\end{proposition} 

In \cite{FEIN1979}, Feinberg provided a geometric interpretation of deficiency:  $\delta = \dim (\text{Ker } Y \cap \text{Im } I_a)$. From this fact and the STLK, the following result follows. 

\begin{proposition}[Cor. 4.12 \cite{FEIN1979}]\label{prop:dimkerYAk}
Let $\mathscr{N}=(\mathscr{S,C,R})$ be a CRN with deficiency $\delta$ and $t$ terminal strong linkage classes. If every linkage class of the CRN is a terminal strong linkage class, then for each $k \in \mathbb{R}^\mathscr{R}_{>0}$,
$$
\dim (\text{Ker } Y A_k) = \delta + t.
$$
\end{proposition}

Throughout the proof, the vector $\log x\in \mathbb{R}^\mathscr{I}$,where $x \in \mathbb{R}_{>0}^\mathscr{I}$, is given by 
$(\log x)_i = \log x_i,  \text{ for all } i \in \mathscr{I}.$ If $x,y \in  \mathbb{R}^\mathscr{I}$, the standard scalar product $x\cdot y \in  \mathbb{R}$ is defined by 
$
x \cdot y = \sum_{i \in \mathscr{I}} x_i y_i.
$ 

\subsection*{Proof of Theorem \ref{th:DZACR}}
Let $c^*$ is a positive steady state of the PL-RDK system. That is, there exists $k \in \mathbb{R}^\mathscr{R}_{>0}$ such that
\begin{equation}\label{eq:c*}
\sum_{y \rightarrow y' \in \mathscr{R}} k_{y \rightarrow y'} (c^*)^{\widetilde{y}} (y'-y)=0.
\end{equation}
Here, we write $\widetilde{y}$ for $\widetilde{Y}_{\cdot,y}$, where $\widetilde{Y}$ is the $m \times n$ matrix defined by M\"{u}ller and Regensburger in \cite{MURE2012} and is constructed as follows: For each reactant complex, the associated column of $\widetilde{Y}$ is the transpose of the kinetic order matrix row of the complex's reaction, otherwise (i.e., for non-reactant complexes), the column is 0. Hence, $\widetilde{y}=\widetilde{Y}_{\cdot,y}$ refers to the kinetic order vector of the reactant complex $y$.
\\
For each $y \rightarrow y' \in \mathscr{R}$, define the positive number $\kappa_{y \rightarrow y'}$ by
\begin{equation}\label{eq:kappa}
\kappa_{y \rightarrow y'}:=k_{y \rightarrow y'}(c^*)^{\widetilde{y}}.
\end{equation}
Thus, we obtain
\begin{equation}\label{eq:1}
\sum_{y \rightarrow y' \in \mathscr{R}} \kappa_{y \rightarrow y'} (y'-y)=0.
\end{equation}
Suppose that $c^{**}$ is also a positive equilibrium of the system. Hence,
\begin{equation}\label{eq:c**}
\sum_{y \rightarrow y' \in \mathscr{R}} k_{y \rightarrow y'} (c^{**})^{\widetilde{y}} (y'-y)=0.
\end{equation}
Define 
\begin{equation}\label{eq:mu}
\mu := \log c^{**} - \log c^*.
\end{equation}
With $\kappa \in \mathbb{R}^\mathscr{R}_{>0}$ given by Equation (\ref{eq:kappa}) and $\mu$ given by Equation (\ref{eq:mu}), it follows from Equation (\ref{eq:c**}) that 
\begin{equation}\label{eq:2}
\sum_{y \rightarrow y' \in \mathscr{R}} \kappa_{y \rightarrow y'} e^{\widetilde{y}\cdot \mu} (y'-y)=0.
\end{equation}
Let $\bm{1}^\mathscr{C} \in \mathbb{R}^\mathscr{C}$ such that 
$$
\bm{1}^\mathscr{C}  = \sum_{y \in \mathscr{C}} \omega_y.
$$
Observe that Equations (\ref{eq:1}) and  (\ref{eq:2}) can be respectively written as
$$
Y A_\kappa \bm{1}^\mathscr{C} =0, \text{ and }
Y A_\kappa \left( \sum_{y \in \mathscr{C}} e^{\widetilde{y}\cdot \mu} \omega_y \right) = 0.
$$
Equivalently, 
\begin{equation} \label{eq: w_c in Ker YAk}
\bm{1}^\mathscr{C} \in \text{Ker }YA_\kappa , \text{ and}
\end{equation}
\begin{equation} \label{eq: Sum in Ker YAk}
\sum_{y \in \mathscr{C}} e^{\widetilde{y}\cdot \mu}  \omega_y \in \text{Ker }YA_\kappa .
\end{equation}
Therefore, $c^*$ and $c^{**}$ are positive equilibria of the PL-RDK system $(\mathscr{N},K)$ if and only if (\ref{eq: w_c in Ker YAk}) and (\ref{eq: Sum in Ker YAk}) hold. \\
\indent Since the network is deficiency zero, its steady states are all complex balanced \cite{FEIN1972}. According to a classical result of Horn \cite{HORN1972}, the underlying network is necessarily weakly reversible. Consequently,  every linkage class of the CRN is a terminal strong linkage class. Moreover, since  $\delta=0$, it follows from Proposition \ref{prop:dimkerYAk} that 
\begin{equation}\label{eq:=t}
\dim (\text{Ker } Y A_\kappa) =t.
\end{equation}
Note that $\text{Ker }A_\kappa \subseteq \text{Ker } YA_\kappa$. But because of Equation (\ref{eq:=t}), we have  $\text{Ker }A_\kappa = \text{Ker } YA_\kappa$. \\
\indent Let $\{ b^1,b^2,\dots,b^t \} \subset \mathbb{R}^\mathscr{C}_{\geq 0}$ be a basis for $\text{Ker }A_\kappa$ as in Proposition \ref{th: STLK} (STLK). Because $\text{Ker }A_\kappa = \text{Ker } YA_\kappa$ and $\bm{1}^\mathscr{C} \in \text{Ker }YA_\kappa$, it must be that
\begin{equation}\label{eq:in Ker Ak}
\bm{1}^\mathscr{C} \in \text{Span } \{b^1,b^2, \dots, b^t \}.
\end{equation}
We consider a new basis for $\text{Ker }YA_\kappa$ that includes $\bm{1}^\mathscr{C}$. Consider removing the vector in $ \{b^1,b^2, \dots, b^t \}$ whose support are precisely those reactant complexes, $y$ and $y'$, with interactions differing only in one species. For convenience, assume that this vector is $b^1$. Hence,  $ \{\bm{1}^\mathscr{C},b^2, \dots, b^t \}$ forms a basis for $\text{Ker }YA_\kappa$.\\
\indent Since $\displaystyle{\sum_{y \in \mathscr{C}}} e^{\widetilde{y}\cdot \mu}  \omega_y \in \text{Ker }YA_\kappa$, there exist $\lambda_1, \lambda_2, \dots, \lambda_t$ such that 
\begin{equation}\label{eq:S27}
\sum_{y \in \mathscr{C}} e^{\widetilde{y}\cdot \mu}  \omega_y = \lambda_1 \bm{1}^\mathscr{C} + \sum_{i=2}^t \lambda_i b^i.
\end{equation}

\noindent Observe that each vector $b^i$, $i=2,\dots,t$, has its support entirely on terminal complexes except for the complexes $y$ and $y'$.  This observation, along with Equation (\ref{eq:S27}), implies that for reactant complexes $y \in \mathscr{C}$ and $y' \in \mathscr{C}$, we have
\begin{equation}\label{eq:S28}
\widetilde{y}\cdot\mu = \widetilde{y'}\cdot\mu .
\end{equation}
Or equivalently, we have
\begin{equation}\label{eq:log}
(\widetilde{y} -\widetilde{y'}) \cdot (\log c^{**} - \log c^*)=0.
\end{equation}
Now, since $y, y' \in \mathscr{C}$ are reactant complexes whose interactions differ only in species $X$, we have
$$
\widetilde{y} -\widetilde{y'} = mX 
$$
for some nonzero $m \in \mathbb{R}$. Thus Equation (\ref{eq:log}) reduces to 
$$
m(\log c^*_X - \log c^{**}_X)= 0.
$$
It follows that 
$$
c^*_X = c^{**}_X.
$$
That is, the system has ACR in species $X$.  
\hfill $\qed$


\begin{thebibliography}{10}

\bibitem{ANDERIES} J. Anderies, S. Carpenter, W. Steffen, J. Rockstr\"{o}m, The topology of non-linear global carbon dynamics: From tipping points to planetary boundaries, \textit{Environ. Res. Lett}. \textbf{8} (2013) 044--048.

\bibitem{AJMSM2015} C. Arceo, E. Jose, A. Mar{\'i}n-Sanguino, E. Mendoza, Chemical reaction network approaches to Biochemical Systems Theory, \textit{Math. Biosci.} \textbf{269} (2015) 135--152.


\bibitem{BARKAI1997} N. Barkai, S. Leibler, Robustness in simple biochemical networks, {\em Nature} \textbf{387} (1997) 913--917.

\bibitem{BLANCHINI2011}
F.~Blanchini, E.~Franco, Structurally robust biological networks, {\em BMC Syst. Biol.} \textbf{5} (2011) 1--14.

\bibitem{CMPY2019}
G. Craciun, S. M\"{u}ller, C. Pantea, P. Yu, A generalization of Birch's theorem and vertex-balanced steady states for generalized mass-action systems, {\em Math. Biosci. Eng.} \textbf{16} (2019) 8243--8267.

\bibitem{CP2008} G.~Craciun, C.~Pantea, Identifiability of chemical reaction networks, {\em J. Math. Chem.} \textbf{44} (2008) 244--259.

\bibitem{DEXTER2015} J.~P. Dexter, T.~Dasgupta, J.~Gunawardena,  Invariants reveal multiple forms of robustness in bifunctional enzyme systems, {\em Integr. Biol.} \textbf{7} (2015) 883--894.

\bibitem{FML2020}
H. Farinas, E. Mendoza, A. Lao,
\newblock Chemical reaction network decompositions and realizations of  S-systems (submitted), \textit{ 
 arXiv: 2003.01503} (2020).

\bibitem{FEIN1972} M. Feinberg, Complex balancing in general kinetic systems, \textit{Arch. Ration. Mech. Anal.} \textbf{49} (1972) 187--194.

\bibitem{FEIN1979}
M. Feinberg, \textit{Lectures on chemical reaction networks}. Notes of lectures given at the Mathematics Research Center of the University of Wisconsin, 1979. Available at https://crnt.osu.edu/LecturesOnReactionNetworks.

\bibitem{FEIN1987} M. Feinberg, Chemical reaction network structure and the stability of complex isothermal reactors I: The deficiency zero and deficiency one theorems, \textit{Chem. Eng. Sci.} \textbf{42} (1987) 2229--2268.

\bibitem{FEIN1995} M. Feinberg, The existence and uniqueness of steady states for a class of chemical reaction networks, \textit{Arch. Ration. Mech. Anal.} \textbf{132} (1995) 311--370.

\bibitem{FLRM2020}
N. Fortun, A. Lao, L. Razon, E. Mendoza, Robustness in power-law kinetic systems with reactant-determined interactions, in: {\em Proceedings of the Japan Conference on Geometry, Graphs and Games 2018}, Lect. Notes Comput. Sci., Springer (in press), 2020.

\bibitem{FMRL2019} N. Fortun, E. Mendoza, L. Razon, A. Lao, A deficiency zero theorem for a class of power law kinetic systems with non-reactant determined interactions, {\em MATCH Commun. Math. Comput. Chem.} \textbf{81} (2019) 621--638.

\bibitem{HORN1972} F. Horn, Necessary and sufficient conditions for complex balancing in chemical kinetics, \textit{Arch. Ration. Mech. Anal.} \textbf{49} (1972) 173--186.

\bibitem{HORNJACK1972}  F. Horn, R. Jackson, General mass action kinetics, \textit{Arch. Ration. Mech. Anal.} \textbf{47} (1972) 187--194.

\bibitem{JOHNSTON2014} M. Johnston, Translated chemical reaction networks, {\em Bull. Math. Biol.} \textbf{76} (2014) 1081--1116.

\bibitem{KITANO2004} H. Kitano, Biological robustness, {\em Nat. Rev. Genet.} \textbf{5} (2004) 826--837.

\bibitem{KITANO2007} H. Kitano, Towards a theory of biological robustness, {\em Mol. Syst. Biol.} \textbf{3} (2007) 1--7.

\bibitem{MURE2012} S. M{\"u}ller, G. Regensburger, Generalized mass action systems: Complex balancing equilibria and sign vectors of the stoichiometric and kinetic-order subspaces, \textit{SIAM J. Appl. Math.} \textbf{72} (2012) 1926--1947.

\bibitem{NEML2019}
A. Nazareno, R. P. Eclarin, E. Mendoza, A. Lao, Linear conjugacy of chemical kinetic systems, {\em Math. Biosci. Eng.} \textbf{16} (2019) 8322--8355.

\bibitem{SCHM2002} R. Schmitz, The Earth's carbon cycle: Chemical engineering course material, \textit{Chem. Engin. Edu. } \textbf{36} (2002) 296--309.


\bibitem{SF2010} G. Shinar and M. Feinberg, Structural sources of robustness in biochemical reaction networks, {\em Science} \textbf{327} (2010) 1389--1391.

\bibitem{SF2011} G. Shinar, M. Feinberg, Design principles for robust biochemical reaction networks: what works, what cannot work, and what might almost work, \textit{Math. Biosci.} \textbf{231} (2011) 39--48. 

\bibitem{TAM2018} D. A. Talabis, C. Arceo, E. Mendoza, Positive equilibria of a class of power-law kinetics, \textit{J. Math. Chem.} \textbf{56} (2018) 358--394.

\bibitem{WIUF2013} C. Wiuf, E. Feliu, Power-law kinetics and determinant criteria for the preclusion of multistationarity in networks of interacting species, \textit{SIAM J. Appl. Dyn. Syst.} \textbf{12} (2013)  1685--1721.

\end{thebibliography}
\end{document}